\documentclass [12pt]{article}
\usepackage{amssymb,amsmath,amsthm}
\usepackage{color}

\DeclareMathOperator{\spn}{span}

\usepackage[cp1250]{inputenc}
\usepackage[active]{srcltx}

\newtheorem{thm}{Theorem}[section]
\newtheorem{prop}[thm]{Proposition}
\newtheorem{lem}[thm]{Lemma}
 \newtheorem{rem}[thm]{Remark}
 \newtheorem{cor}[thm]{Corollary}

\newtheorem{ex}[thm]{Example}
\DeclareMathOperator{\supp}{supp}

\newcommand{\C}{\mathbb{C}}
\newcommand{\R}{\mathbb{R}}

\newcommand{\N}{\mathbb{N}}

\newcommand{\la}{\lambda}

\long\def\comment#1{{}}

 \title{Special N-extremal solutions to indeterminate moment problems}
 \author{Christian Berg and Ryszard Szwarc}

\begin{document}

 \maketitle


\begin{abstract}   For an N-extremal solution $\mu$ to an indeterminate moment problem it is known by a theorem of M. Riesz that the measure $(1+x^2)^{-1}d\mu(x)$ is determinate. For $0<\alpha<1$ we show by contradiction that there exist indeterminate N-extremal solutions $\mu$ such that $(1+x^2)^{-\alpha}d\mu(x)$ is determinate, and there exist also  indeterminate N-extremal solutions $\mu$ such that $(1+x^2)^{-\alpha}d\mu(x)$ is indeterminate. Explicit examples of such measures are so far only known when $\alpha=1/2$.

For indeterminate Stieltjes moment problems and for N-extremal solutions $\mu$,  we show that $(1+x^2)^{-1/2}d\mu(x)$ is indeterminate except when $\mu=\mu_F$ is the Friedrichs solution in case of which $(1+x^2)^{-1/2}d\mu_F(x)$ is determinate. 

We identify the Friedrichs and Krein solutions for some indeterminate Stieltjes moment problems.
\end{abstract}

{\bf Mathematics Subject Classification}: 44A60.

{\bf Keywords}: Indeterminate moment problems, N-extremal solutions, Krein and Friedrichs solutions.

\section{Introduction and main results}

Let $\ell^2$ denote the Hilbert space of square summable complex  sequences $c=(c_n)_{n\ge 0}$. The standard orthonormal basis of $\ell^2$ is denoted $e_n, n=0,1,\ldots$. Furthermore,
\begin{equation}\label{eq:F}
\mathcal F:=\spn\{e_n \mid n\ge 0\}
\end{equation}
is the dense subspace of sequences with only finitely many non-zero terms.

Let $\mathcal M^*$ denote the set of positive measures on $\R$ with infinite support and moments of any order. For $\mu\in\mathcal M^*$ the corresponding moment sequence  $s=(s_n)_{n\ge 0}$ is defined by
\begin{equation}\label{eq:mom}
s_n=\int x^n\,d\mu(x),\quad n\ge 0.
\end{equation}
A sequence $s=(s_n)_{n\ge 0}$ of real numbers is a moment sequence of the form \eqref{eq:mom} if and only if $s$ is positive definite  in the sense that
\begin{equation*}
\sum_{j,k=0}^\infty s_{j+k}c_j\overline{c_k}>0\;\mbox{for all}\;c\in\mathcal F, c\neq 0. 
\end{equation*}
This is the famous theorem of Hamburger, see \cite{Sch}.

A measure $\mu\in\mathcal M^*$ and the corresponding moment sequence $s$ is called {\it determinate} if there is only one measure $\mu\in\mathcal M^*$ such that \eqref{eq:mom}
holds, and {\it indeterminate} if there are more than one such measure. In the indeterminate case the set $V=V(s)$ of measures $\mu\in \mathcal M^*$  satisfying \eqref{eq:mom} is an infinite convex set. 

Let $\mathcal M_d^*$ resp. $\mathcal M_i^*$ denote the sets of determinate resp. indeterminate measures in $\mathcal M^*$.

The following elementary observations are used throughout the paper, and $\C[x]$ denotes the set of polynomials in $x$ with complex coefficients:

\begin{lem}\label{thm:elem} Let $\mu,\nu\in\mathcal M^*$ satisfy $\mu\le \nu$.
\begin{enumerate}
\item[(i)] $\nu\in\mathcal M_d^*\implies \mu\in\mathcal M_d^*$.
\item[(ii)] $\C[x]\,\mbox{dense in}\; L^2(\nu)\implies \C[x]\,\mbox{dense in}\; L^2(\mu)$.
\end{enumerate}
\end{lem}

For $a\in\R$ let $\kappa_a$ denote the translation of $\R$ given by $\kappa_a(x)=x+a$. 

\begin{lem}\label{thm:trans} For $\mu\in\mathcal M^*$ and an arbitrary translation $\kappa_a$ we have
\begin{enumerate}
\item[(i)] $\mu\in\mathcal M_d^* \;(\mbox{resp.}\; \mathcal M_i^*)\implies \kappa_a(\mu)\in\mathcal M_d^*\;(\mbox{resp.}\; \mathcal M_i^*)$.
\item[(ii)] $\C[x]\,\mbox{dense in}\; L^2(\mu)\implies \C[x]\,\mbox{dense in}\; L^2(\kappa_a(\mu))$.
\end{enumerate}
\end{lem}
 
Concerning the proof of (ii) we use that the set of continuous functions with compact support $C_c(\R)$ is dense in $L^2(\mu)$ for any $\mu\in\mathcal M^*$ and that $C_c(\R)$ and $\C[x]$ are invariant under translations.

To $\mu\in\mathcal M^*$  there is a corresponding sequence $(p_n)_{n\ge 0}$ of orthonormal polynomials satisfying
\begin{equation}\label{eq:orth}
\int p_n(x)p_m(x)\,d\mu(x)=\delta_{n,m},
\end{equation}
uniquely determined by \eqref{eq:orth}, if $p_n$ is required to be a polynomial of degree $n$ with positive leading coefficient. Clearly $p_0=s_0^{-1/2}$. In the indeterminate case the polynomials $p_n$ are independent of $\mu\in V$ and they satisfy the three term recurrence relation
\begin{equation}\label{eq:3term}
xp_n(x)=a_np_{n+1}(x)+b_np_n(x)+a_{n-1}p_{n-1}(x), n\ge 0,\quad a_{-1}:=0,
\end{equation}
where
$$
a_n=\int xp_n(x)p_{n+1}(x)\,d\mu(x)>0,\quad b_n=\int xp_n(x)^2\,d\mu(x)\in\R,\quad n\ge 0.
$$
The corresponding Jacobi matrix is the tridiagonal matrix of the  form
\begin{equation}\label{eq:Jac}
J=\begin{pmatrix}
b_0 & a_0 & 0 & \hdots\\
a_0 & b_1 & a_1 & \hdots\\
0 & a_1 & b_2 & \hdots\\
\vdots &\vdots & \vdots & \ddots
\end{pmatrix}.
\end{equation}

It is easy to see that the proportional measures $\la\mu, \la>0$ lead to the same Jacobi matrix $J$, and the well-known Theorem of Favard (see \cite[Theorem 5.14]{Sch}) states that any matrix of the form \eqref{eq:Jac} with $a_n>0, b_n\in\R$ comes from a unique  moment sequence  $(s_n)$ as above, normalized such that $s_0=1$. In the following we shall always assume that this normalization holds, and consequently the solutions $\mu$ of \eqref{eq:mom} are probability measures  and $p_0=1$.

For later use we need the polynomials $(q_n)_{n\ge 0}$ of the second kind defined by
\begin{equation}\label{eq:q}
q_n(z)=\int\frac{p_n(z)-p_n(x)}{z-x}\,d\mu(x),\quad z\in\C,\mu\in V.
\end{equation}
The Jacobi matrix acts as a symmetric operator  in the Hilbert space $\ell^2$ with domain  $\mathcal F$ given in \eqref{eq:F}, and the action is multiplication of the matrix $J$ by $c\in\mathcal F$ considered as a column, i.e.,
\begin{equation}\label{eq:J} 
(Jc)_n:=a_{n-1}c_{n-1}+b_n c_n+a_n c_{n+1},\quad n\ge 0.
\end{equation}

 The Jacobi operator associated with $J$ is by definition the closure $(T,D(T))$ of the symmetric operator $(J,\mathcal F)$.

It is a classical fact that the closed symmetric operator $(T,D(T))$ has deficiency indices  either $(0,0)$  or $(1,1)$. These cases occur precisely if the  moment sequence \eqref{eq:mom} is determinate or indeterminate. See \cite{Ak}, \cite{S}, \cite{Sch} for details. Recent results about indeterminate moment problems can be found in 
\cite{B:S1}, \cite{B:S2}, \cite{B:S3}, \cite{B:S4}.

The domain $D(T)$ of the Jacobi operator has been studied in \cite{B:S5} and \cite{B:S6}.

In the indeterminate case the set $V$ of solutions to \eqref{eq:mom} is given by the Nevanlinna parametrization cf. \cite[Section 7.4]{Sch}  
\begin{equation}\label{eq:Nev}
\int\frac{d\mu_{\varphi}(x)}{x-z}=-\frac{A(z)+\varphi(z)C(z)}{B(z)+\varphi(z)D(z)},\quad z\in\C\setminus\R,
\end{equation}
where $\varphi:\C\setminus\R\to\C$ is a Pick function or $\infty$. When $\varphi$ is a real constant $t$ or $\infty$, we get the so-called N-extremal measures $\mu_t,t\in\R^*:=\R\cup\{\infty\}$. In  \cite{S} they are called von Neumann solutions because they are affilated with the self-adjoint extensions in $\ell^2$ of the Jacobi operator
$(T,D(T))$. The functions $A,B,C,D$ are entire functions defined in terms of the polynomials $p_n,q_n$.

The set of indeterminate and N-extremal measures is denoted $\mathcal M_{i,0}^*$.
We have chosen this symbol because the measures in $\mathcal M_{i,0}^*$ are called
of order 0 in \cite[p. 115]{Sch}.

Marcel Riesz proved the following theorem in 1923, see \cite{Ri}:

\begin{thm}\label{thm:Riesz} Let $\mu\in\mathcal M^*$. The polynomials are dense in $L^2(\mu)$ if and only if the measure $(1+x^2)^{-1}d\mu(x)$ is determinate and if and only if $\mu$ is either determinate or indeterminate and N-extremal.
\end{thm}

We summarize some of the properties of the N-extremal measures $\mu_t$, which can be found in \cite{Ak} and \cite{Sch}.

\begin{prop}\label{thm:supportNext}
(i) The  solution $\mu_t$ is a discrete measure with support equal to the countable zero set $\Lambda_t$ of the entire function $B(z)+tD(z)$, with the convention that $\Lambda_\infty$ is the zero set of $D$. We have $\Lambda_t\subset\R$ for $t\in\R^*$.

(ii) The support of two different N-extremal solutions are disjoint and interlacing. Each point $x_0\in\R$ belongs to the support of a unique N-extremal measure $\mu_t$, where
$t\in\R^*$ is given  as $t=-B(x_0)/D(x_0)$ if $D(x_0)\neq 0$ and $t=\infty$ if $D(x_0)=0$.

(iii) If $x_0\in\R$ belongs to the support of the N-extremal measure $\mu_t$, then the measure $\mu_t-\mu_t(x_0)\delta_{x_0}$ is determinate.
\end{prop}

Putting $z=0$ in \eqref{eq:Nev}, which is possible when $0\notin \supp(\mu_t)$, leads to
\begin{equation}\label{eq:Npar0}
\int\frac{d\mu_t(x)}{x}=t,\quad t\in\R.
\end{equation}

We fix a number $0\le\alpha\le 1$ and divide $\mathcal M_i^*$  in two disjoint subsets:
\begin{eqnarray}\label{eq:A(alpha)}
A(\alpha)&:=&\{\mu\in \mathcal M_i^* \mid (1+x^2)^{-\alpha}d\mu(x) \in \mathcal M_i^*\},\\
\label{eq:B(alpha)}
B(\alpha)&:=&\{\mu\in \mathcal M_i^* \mid (1+x^2)^{-\alpha}d\mu(x) \in\mathcal M_d^*\}.
\end{eqnarray}
We have
$$
A(0)=\mathcal M_i^*,\; B(0)=\emptyset,\; A(1)=\mathcal M_i^*\setminus\mathcal M_{i,0}^*,\;
B(1)=\mathcal M_{i,0}^*.
$$
By Lemma~\ref{thm:elem} we clearly have: 
\begin{equation}\label{eq:incl}
0\le\alpha_1<\alpha_2\le 1\implies
A(\alpha_2)\subseteq A(\alpha_1),\; B(\alpha_1)\subseteq B(\alpha_2).
\end{equation}

Furthermore, for $\mu\in\mathcal M_{i,0}^*$ the following defining equation for $a(\mu)$ holds:
\begin{equation}\label{eq:a}
a(\mu)=\sup\{\alpha\in [0,1] \mid \mu\in A(\alpha)\}=\inf\{\beta\in [0,1]\mid \mu\in B(\beta)\}.
\end{equation}

There is no apparent relation between the moments of $\mu$ and $(1+x^2)^{-\alpha}d\mu(x)$ when $0<\alpha<1$.

\begin{prop}\label{thm:shift} Let $\mu\in\mathcal M_{i,0}^*$.

(a) Assume $a(\mu)<1$. For $0<\delta<1-a(\mu)$ we have
\begin{equation}\label{eq:a1}
a((1+x^2)^\delta d\mu(x))=a(\mu)+\delta.
\end{equation}

(b) Assume $0<a(\mu)$. For $0<\delta<a(\mu)$ we have
\begin{equation}\label{eq:b1}
a((1+x^2)^{-\delta} d\mu(x))=a(\mu)-\delta.
\end{equation}
\end{prop}

\begin{proof} $(a)$: Note that $(1+x^2)^\delta d\mu(x)\in\mathcal M_{i,0}^*$. It is indeterminate because it majorizes $\mu$. Choosing $\varepsilon>0$ so small that $\varepsilon+\delta<1-a(\mu)$ we know that $(1+x^2)^{-(a(\mu)+ \varepsilon)}d\mu(x)$ is determinate, and then $\C[x]$ is dense in $L^2((1+x^2)^{1-a(\mu)-\varepsilon}d\mu(x))$ and a fortiori in $L^2((1+x^2)^\delta d\mu(x))$.

It is now easy to establish \eqref{eq:a1}.

$(b)$: Note that $(1+x^2)^{-\delta}d\mu(x)\in\mathcal M_{i,0}^*$ and the proof is analogous to the proof of (a). 
\end{proof}

The following fundamental result is proved in Section 2.
\begin{thm}\label{thm:alpha} For  $0<\alpha<1$ we have $A(\alpha)\cap {\mathcal M_{i,0}^*}\neq\emptyset, \emptyset\neq B(\alpha)\subseteq {\mathcal M_{i,0}^*}$.
\end{thm}

Using the quantity $a(\mu)$ we define the sets
\begin{eqnarray}\label{eq:A,B}
A&=&\{\mu\in\mathcal M_{i,0}^*\mid (1+x^2)^{-a(\mu)}d\mu(x)\in\mathcal M_i^*\}\\
B&=&\{\mu\in\mathcal M_{i,0}^*\mid (1+x^2)^{-a(\mu)}d\mu(x)\in\mathcal M_d^*\},
\end{eqnarray}
and clearly
$$
A\cup B=\mathcal M_{i,0}^*,\quad A\cap B=\emptyset.
$$

The following result is proved in Section 2.

\begin{prop} \label{thm:nonempty}
(i) $A\neq \emptyset$ if and only if there exists $\nu\in\mathcal M_{i,0}^*$ such that $(1+x^2)^{-\delta}d\nu(x)\in\mathcal M_d^*$ for any $\delta>0$.

(ii) $B\neq\emptyset$ if and only if  there exists $\nu\in\mathcal M_d^*$ such that $(1+x^2)^{\delta}d\nu(x)\in\mathcal M_i^*$ for any $\delta>0$.
\end{prop}

\medskip
{\bf Open Problem:} Are both sets $A,B$ nonempty?

\section{Proofs}

Proof of Theorem~\ref{thm:alpha}:

 If we assume that $A(\alpha)\cap {\mathcal M_{i,0}^*}=\emptyset$ for some $0<\alpha<1$, then for all $\mu\in\mathcal M_{i,0}^*$ we have $(1+x^2)^{-\alpha}d\mu(x)\in\mathcal M _d^*$. By Theorem~\ref{thm:Riesz} we get that the polynomials are dense in $L^2((1+x^2)^{1-\alpha}d\mu(x))$ and $(1+x^2)^{1-\alpha}d\mu(x)$ is also indeterminate since it is greater than $\mu$, hence N-extremal. Repeating this procedure $k$ times we get that the measure $(1+x^2)^{k(1-\alpha)}d\mu(x)$
is indeterminate and N-extremal so
$\nu_k:=(1+x^2)^{k(1-\alpha)-1}d\mu(x)$ is determinate  by Theorem~\ref{thm:Riesz}.
For $k$ so large that $k(1-\alpha)\ge 1$ this is impossible since $\nu_k$ majorizes $\mu$.

\medskip
If $B(\alpha)=\emptyset$ for some $0<\alpha<1$, then for all $\mu\in\mathcal M_i^*$
we have $(1+x^2)^{-\alpha}d\mu(x)\in\mathcal M_i^*$. 
Repeating this $k$ times we get that $(1+x^2)^{-k\alpha}d\mu(x)$
is indeterminate. When $k\alpha\ge 1$ we get
$$
(1+x^2)^{-k\alpha}d\mu(x)\le (1+x^2)^{-1}d\mu(x),
$$ 
so also the measure $(1+x^2)^{-1}d\mu(x)$ is indeterminate. This however contradicts Theorem~\ref{thm:Riesz}, if we start with an indeterminate N-extremal measure $\mu$.

A measure $\mu\in B(\alpha)$ is necessarily N-extremal because $(1+x^2)^{-1}d\mu(x)$ is majorized by the determinate measure $(1+x^2)^{-\alpha}d\mu(x)$ and hence determinate. $\square$

\medskip
Proof of Proposition~\ref{thm:nonempty}:

(i) If $\mu\in A$ then $\nu:=(1+x^2)^{-a(\mu)}d\mu(x)$ is indeterminate by definition, and it is N-extremal as it is majorized by $\mu$. It has clearly the property that $(1+x^2)^{-\delta}d\nu(x)\in \mathcal M_d^*$ for any $\delta>0$.

If $\nu$ has  the stated property, we have in particular   $(1+x^2)^{-1}d\nu(x)\in\mathcal M_d^*$, so the polynomials are dense in $L^2(\nu)$ by Theorem~\ref{thm:Riesz} and furthermore $a(\nu)=0$. This shows that $A\neq \emptyset$ since $\nu\in A$. 

\medskip
(ii) If $\mu\in B$ then $a(\mu)>0$ and $\nu:=(1+x^2)^{-a(\mu)}d\mu(x)\in\mathcal M_d^*$ has the property that for any $0<\delta< a(\mu)$
$$
(1+x^2)^{\delta}d\nu(x)=(1+x^2)^{-(a(\mu)-\delta)}d\mu(x)\in \mathcal M_i^*, 
$$
and $(1+x^2)^\delta d\nu(x)$ is also indeterminate when $\delta\ge  a(\mu)$, since it is larger than $\mu$.

On the other hand, if $\nu$  has the stated property, then in particular $\tau:=(1+x^2)d\nu(x)\in \mathcal M_i^*$  and it is N-extremal by Theorem~\ref{thm:Riesz}. Furthermore $a(\tau)=1$, so we have $\tau\in B$ because
$(1+x^2)^{-a(\tau)}d\tau(x)=\nu$. $\square$

Using Stieltjes problems we shall find measures in $A(\frac12)$ and $B(\frac12)$. Unfortunately we have not been able to find explicit measures in $A(\alpha)$ and $B(\alpha)$ when $\alpha\in(0,1)\setminus\{\frac12\}$.

\section{The Stieltjes case}

Let $s$  be a normalized indeterminate moment sequence. Introducing 
\begin{equation}\label{eq:is}
\xi(\mu):=\inf\supp(\mu), \quad \mu\in V,
\end{equation}
then $\xi(\mu)\in[-\infty,\infty)$. We say that $s$ is a Stieltjes moment sequence 
if there exists a measure $\mu\in V$ with $\supp(\mu)\subseteq [0,\infty[$, so $s$ is Stieltjes iff $\xi(\mu)\ge 0$ for at least one $\mu\in V$.

In the following we assume $s$ to be a Stieltjes moment sequence.

The indeterminate moment sequence $s$ can be determinate as a Stieltjes moment sequence (det(S) in short), if there is only one measure $\mu\in V$ supported by $[0,\infty)$, and indeterminate as a Stieltjes moment sequence (indet(S) in short), if there are more than one measure $\mu\in V$ supported by $[0,\infty)$.

In the det(S) case the N-extremal measure $\mu_{\infty}$ is the unique solution supported by $[0,\infty)$. 

As explained in \cite{B:V}  and repeated in \cite{B} (and going back to Chihara with another terminology and proof), there is a quantity $\alpha(s)\le 0$ defined by
\begin{equation}\label{eq:al}
\alpha(s):=\lim_{n\to\infty} \frac{p_n(0)}{q_n(0)}=\lim_{x\to -\infty} \frac{D(x)}{B(x)}.
\end{equation}
The problem is det(S) iff $\alpha(s)=0$ and indet(S) iff
$\alpha(s)<0$.

Since the Nevanlinna parametrization used here is related to the Nevanlinna parametrization used in \cite{Ak} by replacing $t$ by $-1/t$ we define
\begin{equation}\label{eqFr}
F(s):=\left\{\begin{array}{cl}-1/\alpha(s),& \alpha(s)<0,\\
\infty,& \alpha(s)=0.
\end{array}
\right.
\end{equation}
We then get that the indeterminate N-extremal solutions $\mu_t$ given in \eqref{eq:Nev} are supported on $[0,\infty[$  iff
$ F(s)\le t\le\infty$. For a proof see \cite{B} and references therein.

In the indet(S) case the solution $\mu_\infty$ is called the Krein solution and we also denote it $\mu_{K}$. It is concentrated in the zeros $0=d_1<d_2<\cdots$ of $D$. The measure $\mu_F=\mu_{F(s)}$ is called the Friedrichs solution, because as proved in \cite{Pe1}, it corresponds to the Friedrichs extension of the Jacobi matrix. See details in \cite{B} with further references. The following is proved there:

\begin{prop}\label{thm:revisit}
$(i)$: $\mu_F$ is concentrated in the sequence $\xi_1<\xi_2<\cdots$ defined by
 \begin{equation}\label{eq:limzero}
\xi_k=\lim_{n\to\infty} x_{n,k},\quad k=1,2,\ldots,
\end{equation}
 where $x_{n,k}, k=1,2,\ldots,n$ denote
the $n$ zeros in increasing order of the polynomial $p_n$.    

$(ii)$: $\xi_1=\xi(\mu_{F(s)})>0$ and for all solutions $\mu\in V$ to the indeterminate Hamburger problem we have $\xi(\mu)<\xi_1$ unless $\mu=\mu_F$.

$(iii)$: When $t$ increases from  $F(s)$ to $\infty$, then 
$\xi(\mu_t)$ decreases from  $\xi_1$ to $0=\xi(\mu_K)$. 

$(iv)$: For $t<F(s)$ then $\mu_t$ has one negative mass-point equal to $\xi(\mu_t)$, and when $t$ increases from $-\infty$ to $F(s)$ then $\xi(\mu_t)$ decreases from $0$ to $-\infty$.
\end{prop}

\medskip
In the det(S) case  we may say that the Krein and Friedrichs solutions coincide with the N-extremal solution  $\mu_\infty$ supported in the zeros of $D$.

\begin{rem}{\rm The sequence $(\xi_n)$ is studied in \cite{Ch} without discussion of the Friedrichs solution. 

Notice that $F(s)$ equals the Friedrichs parameter $\gamma_s$ defined in formula (8.35) in \cite{Sch}. 

We can of course also discuss the Friedrichs  and Krein solutions for a non-normalized indeterminate Stieltjes moment sequence.}
\end{rem}

\begin{thm}\label{thm:-1} Let $s$ be a normalized Stieltjes moment sequence which is indet(S). 

(a) For $F(s)<t<\infty$ the measure $x^{-1}d\mu_t(x)$ is indet(S), while $x^{-1}d\mu_{F(s)}(x)$ is determinate. 

(b) For $F(s)<t<\infty$ the measure $x^{-1}d\mu_t(x)$ is the Friedrichs solution for the moment sequence $s^{(-1)}[t]$ given by
\begin{equation}\label{eq:-1mom}
s^{(-1)}[t](n)=\left\{\begin{array}{ll} t, & n=0, \\ s_{n-1}, & n\ge 1, \end{array}\right.
\end{equation}
while $\tau_t:=(t-F(s))\delta_0 + x^{-1}d\mu_{F(s)}(x)$ is the Krein solution for \eqref{eq:-1mom}.

(c) For $F(s)<t'<t<\infty$ the measure $(t-t')\delta_0+x^{-1}d\mu_{t'}(x)$ is a non N-extremal solution to \eqref{eq:-1mom}.
\end{thm} 

\begin{proof} $(a)$: For $F(s)<t<\infty$ we have by \eqref{eq:Npar0} 
that the measures $\tau_t=(t-F(s))\delta_0+x^{-1}d\mu_{F(s)}(x)$ and $x^{-1}d\mu_t(x)$ have the same total mass $t$ and the moment sequence \eqref{eq:-1mom}. They are clearly different and in particular $x^{-1}d\mu_t(x)$ is indet(S).

Assume that $x^{-1}d\mu_{F(s)}(x)$ is indeterminate, which will lead to  a contradiction. Let $\sigma_K$ denote the Krein solution of the corresponding moment sequence $s^{(-1)}[F(s)]$ given by \eqref{eq:-1mom} when $t=F(s)$. 

The supports of $x^{-1}d\mu_{F(s)}(x)$ and $\sigma_K$ are  $(\xi_n)_{n\ge 1}$ given in \eqref{eq:limzero} and a sequence $(\eta_n)_{n\ge 1}$ satisfying 
$0=\eta_1<\xi_1<\eta_2<\xi_2<\cdots$ by Proposition~\ref{thm:supportNext} (ii).

The measures $xd\sigma_K(x)$ and $\mu_{F(s)}$ have the same moments. Furthermore,  $\sigma_K-\sigma_K(0)\delta_0$ is determinate by (iii) in Proposition~\ref{thm:supportNext}. By Theorem~\ref{thm:Riesz} 
the polynomials are dense in $L^2$ with respect to $(1+x^2)d[\sigma_K-\sigma_K(0)\delta_0](x)$, and hence also in $L^2$ with respect to
$$
xd[\sigma_K-\sigma_K(0)\delta_0](x)=xd\sigma_K(x).
$$
This shows that $xd\sigma_K(x)$ is an indeterminate N-extremal solution to the moment sequence $s$, and we have $\eta_2=\xi(xd\sigma_K(x))>\xi_1=\xi(\mu_{F(s)})$,
which contradicts that the Friedrichs solution has the greatest lower bound among the solutions, see Proposition~\ref{thm:revisit} (ii).

$(b)$: Since $x^{-1}d\mu_t(x)\le \xi(\mu_t)^{-1}d\mu_t(x)$, we see that the polynomials are dense in $L^2(x^{-1}d\mu_t(x))$, so $x^{-1}d\mu_t(x)$ is an N-extremal solution for
\eqref{eq:-1mom}.

The measure $(1+x^2)^{-1}d\mu_t(x)$ is determinate by Theorem~\ref{thm:Riesz} and so is $x^{-2}d\mu_t(x)$, since $(1+x^2)/x^2$ is bounded above on $\supp(\mu_t)$. By part (a) it follows that $x^{-1}d\mu_t(x)$ is the Friedrichs solution for \eqref{eq:-1mom}.   

We finally need to show that $\tau_t$ is N-extremal. This follows from Lemma 2 p. 113 in \cite{B:C}. For the benefit of the reader we give a self-contained proof. Assume for contradiction that $\tau_t$ is not N-extremal for \eqref{eq:-1mom}. Then there exists an N-extremal solution $\kappa$ with $\kappa(0)>t-F(s)$ by \cite[Theorem 7.15]{Sch}. If we subtract the measure $(t-F(s))\delta_0$ from $\kappa$ and $\tau_t$ we get two measures $\kappa-(t-F(s))\delta_0$
and $x^{-1}d\mu_{F(s)}(x)$ with the same moments and they are equal because the latter is determinate. However, the two measures cannot be equal as the first has mass at 0 and the second not.  

$(c)$: The given measure has the moments \eqref{eq:-1mom}, but since the mass at zero is less than $t-F(s)$ it cannot be N-extremal. 
\end{proof}

\begin{rem}\label{thm:denind}{\rm The authors of \cite{B:T2} introduced a density index for a measure $\sigma$ on $[0,\infty)$ with moments of all orders and for which  the polynomials are dense in $L^2(\sigma)$, namely
\begin{equation}\label{eq:denind}
\delta(\sigma):=\sup\{k\in\N_0 \mid \C[x]\;\mbox{dense in}\, L^2(x^kd\sigma(x))\}.
\end{equation}
It was proved that for the N-extremal solutions $\mu_t, F(s)\le t\le \infty$ supported by $[0,\infty)$ one has
\begin{equation}\label{eq:denind1}
\delta(\mu_{F(s)})=1,\;\delta(\mu_{\infty})=2,\;\delta(\mu_t)=0\;\mbox{for}\; F(s)<t<\infty.
\end{equation}

Theorem~\ref{thm:-1} (a) can be deduced from these results about the density index in the following way:

From $\delta(\mu_t)=0$ we know that $\C[x]$ is not dense in $L^2(xd\mu_t(x))$, but then
$x/(1+x^2)d\mu_t(x)$ is indeterminate as well as $(1/x)d\mu_t(x)$ because $x^2/(1+x^2)$ is bounded above and below on $\supp(\mu_t)\subset [\xi(\mu_t),\infty)$.

From $\delta(\mu_{F(s)})=1$ we get that $x/(1+x^2)d\mu_{F(s)}(x)$ is determinate and so is $(1/x)d\mu_{F(s)}(x)$.
}
\end{rem}

\begin{cor}\label{thm:Es} Let $s$ be a normalized Stieltjes moment sequence which is
indet(S). Then the measures  $xd\mu_t(x), F(s)\le t\le \infty$ are solutions to the shifted moment sequence $(s_{n+1})_{n\ge 0}$.

The measure $xd\mu_{F(s)}(x)$ is an N-extremal solution, $xd\mu_\infty(x)=xd\mu_K(x)$ is the Friedrichs solution. For $F(s)<t<\infty$ the measures $xd\mu_t(x)$ are non N-extremal solutions.
\end{cor}

\begin{proof} From \eqref{eq:denind1} we get
$$
\delta(xd\mu_{F(s)}(x))=0,\;\delta(xd\mu_{\infty}(x))=1,
$$
while the polynomials are not dense in $L^2(xd\mu_t(x))$ when $F(s)<t<\infty$.
\end{proof}

From the Stieltjes moment sequence $s$ of Theorem~\ref{thm:-1} we derive a family of Stieltjes moment sequences and determine their  Friedrichs solutions.

\begin{thm}\label{thm:newseq} Let $s$ be a normalized Stieltjes moment sequence  which is indet(S). For $t\in [F(s),\infty)$ let $c_t=\xi(\mu_t)$ denote the smallest mass point in the support of $\mu_t$.

The sequence $\tilde{s}(t)$ given by
\begin{equation}\label{eq:1mom}
\tilde{s}(t)_n=s_{n+1}-c_t s_n,\quad n\ge 0
\end{equation}
is a Stieltjes moment sequence which is determinate when $t=F(s)$ and  indet(S) when $t>F(s)$.

(a) For $t=F(s)$ the unique solution to the moments \eqref{eq:1mom} is
\begin{equation}\label{eq:1momdet}
(x-\xi_1)d\mu_{F(s)}(x),
\end{equation}
 where $\xi_1=\xi(\mu_{F(s)})$, cf. \eqref{eq:limzero}.

(b) For $F(s)<t<\infty$ and $t'\in [F(s),t]$ the measures
\begin{equation}\label{eq:nu}
\nu_{t'}:=(x-c_t)d\mu_{t'}(x)
\end{equation}
are solutions to the moments \eqref{eq:1mom} and
$\nu_t=(x-c_t)d\mu_t(x)$ is the Friedrichs solution. The measure $\nu_{F(s)}$ is N-extremal while the measures $\nu_{t'}$ are non N-extremal when $F(s)<t'<t$. 
\end{thm}

\begin{proof} Assume first $t=F(s)$. The measure \eqref{eq:1momdet} is positive and has the moments \eqref{eq:1mom}.

By Corollary~\ref{thm:Es} the measure $xd\mu_{F(s)}(x)$ is indeterminate and N-extremal. By Proposition~\ref{thm:supportNext} (iii) the measure
$$
xd\mu_{F(s)}(x)-\xi_1\mu_{F(s)}(\xi_1)\delta_{\xi_1}
$$ 
is determinate. We have
$(x-\xi_1)d\mu_{F(s)}(x)\le xd\mu_{F(s)}(x)$. As the measure on the left vanishes at $\xi_1$
we get
$$
(x-\xi_1)d\mu_{F(s)}(x)\le xd\mu_{F(s)}(x)-\xi_1\mu_{F(s)}(\xi_1)\delta_{\xi_1},
$$
and therefore $(x-\xi_1)d\mu_{F(s)}(x)$ is determinate.

\medskip
Assume next $F(s)<t$ and $F(s)\le t'\le t$. 

Since $c_{t'}\ge c_t$  the measures 
$\nu_{t'}$ are positive with the moment sequence \eqref{eq:1mom}, showing that it is a Stieltjes moment sequence which is indet(S).

The measure $\mu_t-\mu_t(c_t)\delta_{c_t}$ is determinate by Proposition~\ref{thm:supportNext} (iii), and since 
 $$
x^{-1}d\nu_t(x)=x^{-1}(x-c_t)d\mu_t(x)=x^{-1}(x-c_t)(d\mu_t(x)-\mu_t(c_t)\delta_{c_t})
\le \mu_t-\mu_t(c_t)\delta_{c_t},
$$
 we see that the measure $x^{-1}d\nu_t(x)$ is determinate. By Theorem~\ref{thm:-1} (a) we get that $\nu_t$ is the Friedrichs solution to the moments \eqref{eq:1mom}.

The properties of the measures $\nu_{t'}$  for $F(s)\le t'<t$ are seen as in 
Corollary~\ref{thm:Es}. 
\end{proof}

\begin{rem}\label{thm:infty} {\rm Theorem~\ref{thm:newseq} for $t=\infty$ reduces to 
Corollary~\ref{thm:Es} because $c_{\infty}=0$.
}
\end{rem}

We now replace $x^{-1}$ by $(1+x^2)^{-1/2}$ in front of the measures $\mu_t$.  Note that $(1+x^2)^{-1/2}d\mu_t(x)\in \mathcal M^*$ also for $-\infty < t<F(s)$, where $\mu_t$ has a negative mass-point.

\begin{thm}\label{thm:(1+x^2)^{-1/2}} Let $s$ be a normalized indeterminate moment sequence which is Stieltjes.

(a) If the problem is indet(S), then the measure $(1+x^2)^{-1/2}d\mu_{F(s)}(x)$ is determinate while
 the measures $(1+x^2)^{-1/2}d\mu_t(x)$  are indet(S) and N-extremal for $F(s)<t <\infty$.

 The measure $(1+x^2)^{-1/2}d\mu_K(x)$ is indeterminate, N-extremal  and det(S). 

(b) If the problem is det(S) and hence $F(s)=\infty$ with the only Stieltjes solution $\mu=\mu_{F(s)}$, then $(1+x^2)^{-1/2}d\mu(x)$ is determinate.

(c) The measures $(1+x^2)^{-1/2}d\mu_t(x)$ are indeterminate and N-extremal for $-\infty<t<F(s)$.

In particular $\mu_F\in B(\frac12)$, hence $a(\mu_F)\le 1/2$, while $\mu_t\in A(\frac12)$, hence $a(\mu_t)\ge 1/2$ for $t\in\R^*\setminus\{F(s)\}$.
\end{thm}

\begin{proof} $(a)$: Assume first that the problem is indet(S). The function $h(x)=\sqrt{1+x^2}/x$ is decreasing on the interval $(0,\infty)$  and hence for $c>0$
$$
1<h(x)\le h(c),\quad c\le x<\infty.
$$
For $F(s)\le t<\infty$ we have using \eqref{eq:is}
$$
(1+x^2)^{-1/2}d\mu_t(x)\le x^{-1}d\mu_t(x)\le h(\xi(\mu_t))(1+x^2)^{-1/2}d\mu_t(x),
$$
and it follows from Theorem~\ref{thm:-1} that the measures $(1+x^2)^{-1/2}d\mu_t(x)$ are indet(S) and N-extremal for $F(s)<t<\infty$ and determinate for $t=F(s)$.

Let $\kappa_1:\R\to\R$ denote the translation $\kappa_1(x)=x+1$. The image measures
$\kappa_1(\mu_t)$ are supported by $[1,\infty)$, when $F(s)\le t\le\infty$ with moments 
$$
\tilde{s}_n=\int (x+1)^n\,d\mu_t(x)=\sum_{k=0}^n\binom{n}{k}s_k,
$$
 and $\tilde{s}$ is indet(S) with Friedrichs solution $\kappa_1(\mu_{F(s)})$, but $\kappa_1(\mu_K)$ is not the Krein solution, since 0 is not in the support. Therefore
$y^{-1}d\kappa_1(\mu_K)(y)$ is indet(S) and 
\begin{equation}\label{eq:new1}
\kappa_{-1}\left(y^{-1}d\kappa_1(\mu_K)(y)\right)=(1+x)^{-1}d\mu_K(x)
\end{equation} 
is indeterminate. However 
$$
(1+x)^{-1}d\mu_K(x)\le (1+x^2)^{-1/2}d\mu_K(x),
$$
so also $(1+x^2)^{-1/2}d\mu_K(x)$ is indeterminate.

We claim that it is also det(S) and hence N-extremal. Indeed, assume for contradiction that it is indet(S) and let $\sigma_F$ denote the corresponding Friedrichs solution. The measures $xd\sigma_F(x)$ and $x(1+x^2)^{-1/2}d\mu_K(x)$ have the same moments, and they are different since their supports are disjoint. However, we have
$$
x(1+x^2)^{-1/2}d\mu_K(x)=x(1+x^2)^{-1/2}d[\mu_K-\mu_K(0)\delta_0](x)
\le \mu_K-\mu_K(0)\delta_0,
$$
and since the last measure is determinate by (iii) in Proposition~\ref{thm:supportNext},   
so is also $x(1+x^2)^{-1/2}d\mu_K(x)$, which is a contradiction.

$(b)$: Let us next assume that $s$ is det(S) and let $\mu$ be the only Stieltjes solution. In this case $\mu$ is N-extremal concentrated in the zeros of $D$. Since $\xi(\nu)<0$ for all solutions $\nu\in V\setminus \{\mu\}$, the measure $\kappa_1(\mu)$ is indet(S), and it must be the  Friedrichs solution for the moments $\tilde{s}$. By Theorem~\ref{thm:-1} $y^{-1}d\kappa_1(\mu)(y)$ is determinate and so is the translate under $\kappa_{-1}$, which by \eqref{eq:new1} equals 
$(1+x)^{-1}d\mu(x)$. As 
$$
(1+x^2)^{-1/2}\le \sqrt{2}(1+x)^{-1},\quad x\ge 0,
$$
we get that $(1+x^2)^{-1/2}d\mu(x)$ is determinate.

$(c)$: Consider the N-extremal measure $\mu_t$ for $-\infty<t<F(s)$  so $\xi(\mu_t)<0$. Put $a=-\xi(\mu_t)$. Then the translated measure $\kappa_a(\mu_t)$ is supported by $[0,\infty)$ with mass at  0. It is the Krein solution of an indeterminate Stieltjes moment problem, and therefore $(1+x^2)^{-1/2}d\kappa_a(\mu_t)(x)$ is indeterminate, N-extremal and det(S) by part $(a)$. By Lemma~\ref{thm:trans}
$$
\kappa_{-a}\left((1+x^2)^{-1/2}d\kappa_a(\mu_t)(x)\right)=(1+(x+a)^2)^{-1/2}d\mu_t(x)
$$
is indeterminate and N-extremal and so is $(1+x^2)^{-1/2}d\mu_t(x)$.
\end{proof}

\begin{thm}\label{thm:StoF} Let $s$ be a normalized indeterminate moment sequence which is Stieltjes and let $\mu\in V$ be a solution supported on $[0,\infty)$. Assume that the measure $(1+x^2)^{-1/2}d\mu(x)$ is indeterminate. Then it is the Friedrichs solution of the corresponding moment sequence if and only if $\mu$ is N-extremal.
\end{thm}

\begin{proof} Assume first that $(1+x^2)^{-1/2}d\mu(x)$ is the Friedrichs solution of the corresponding moment problem assumed to be indeterminate. By Theorem~\ref{thm:(1+x^2)^{-1/2}} the measure $(1+x^2)^{-1}d\mu(x)$ is determinate and therefore $\mu$ is N-extremal by Theorem~\ref{thm:Riesz}.

Assume next that $\mu$ is N-extremal. Then $c:=\xi(\mu)$ is the smallest mass-point of $\mu$. By assumption $(1+x^2)^{-1/2}d\mu(x)$ is indeterminate, and it is also N-extremal since it is majorized by $\mu$. 
Assume for contradiction that $(1+x^2)^{-1/2}d\mu(x)$ is not the Friedrichs  solution. Thus there is an N-extremal  measure $\sigma$ with $\xi(\sigma)>c$, and  the moments
of $(1+x^2)^{-1/2}d\mu(x)$ and $\sigma$ coincide. Therefore the moments of
$(x-c)(1+x^2)^{-1/2}d\mu(x)$ and $(x-c)d\sigma(x)$ coincide. However, 
$$
(x-c)(1+x^2)^{-1/2}d\mu(x)=(x-c)(1+x^2)^{-1/2}d[\mu-\mu(c)\delta_c](x)\le \mu-\mu(c)\delta_c,
$$
and we conclude that $(x-c)(1+x^2)^{-1/2}d\mu(x)$ is determinate. We get a contradiction as the supports of that measure and $(x-c)d\sigma(x)$ are disjoint. 
\end{proof}

If we apply Theorem~\ref{thm:StoF} to $\mu_t, F(s)<t<\infty$ from Theorem~\ref{thm:(1+x^2)^{-1/2}} we get:

\begin{cor}\label{thm:Fr} Let $s$ be a normalized indeterminate Stieltjes moment sequence which is indet(S). For $F(s)<t<\infty$ the measure $(1+x^2)^{-1/2}d\mu_t(x)$ is the Friedrichs solution of the corresponding moment problem. 
\end{cor}

\begin{thm}\label{thm:fund} Let $s$ be a normalized indeterminate moment sequence which is Stieltjes. 
Let $f$ be a bounded non-negative  Borel function on $[0,\infty)$ such that the measure $f(x)d\mu_{F(s)}(x)$ is indeterminate. Then it is the Friedrichs solution of the corresponding moment problem.  
\end{thm}

\begin{proof} Let $C=\sup_{x\ge 0}f(x)$. By Theorem~\ref{thm:(1+x^2)^{-1/2}} the measure $(1+x^2)^{-1/2}d\mu_{F(s)}(x)$ is determinate. So is $(1+x^2)^{-1/2}f(x)d\mu_{F(s)}(x)$ as
$$
(1+x^2)^{-1/2}f(x)d\mu_{F(s)}(x)\le C(1+x^2)^{-1/2}d\mu_{F(s)}(x).
$$
The measure $(1+x^2)^{1/2}f(x)d\mu_{F(s)}(x)$ is larger than $f(x)d\mu_{F(s)}(x)$ and hence indeterminate, and it is N-extremal by Theorem~\ref{thm:Riesz}. By Theorem~\ref{thm:StoF}  we conclude that $f(x)d\mu_{F(s)}(x)$ is the Friedrichs measure of the corresponding moment problem. 
\end{proof}

In analogy with Theorem~\ref{thm:alpha} we have the following result about indeterminate Stieltjes problems.

\begin{thm}\label{thm:alphaS} Let $0<\alpha<1$ be given. There exists a Stieltjes moment sequence $s$ which is indet(S) and an N-extremal solution $\mu$ with $\xi(\mu)>0$ such that $x^{-\alpha}d\mu(x)$ is determinate. Such a measure $\mu$ is the Friedrichs solution and belongs to $B(\alpha/2)$. In particular $a(\mu)\le \alpha/2$.
\end{thm}

\begin{proof} Assuming that the result does not hold, we get that for any indet(S) moment problem and any N-extremal solution $\mu$ with $\xi(\mu)>0$ the following is true:
$$
x^{-\alpha}d\mu(x) \;\mbox{is indeterminate},
$$
and necessarily indet(S) since $\xi(x^{-\alpha}d\mu(x))>0$.

As $x^{-\alpha}$ is bounded above on $[c,\infty)$ for any $c>0$, the measure $x^{-\alpha}d\mu(x)$ is also N-extremal with the same support as $\mu$. By iteration we get that $x^{-k\alpha}d\mu(x)$ is indeterminate and N-extremal for any $k\in\N$. For $k$ so large that $k\alpha\ge 2$ this implies that $(1+x^2)^{-1}d\mu(x)$ is indeterminate in contradiction to Theorem~\ref{thm:Riesz}.

Since $x^{-1}d\mu(x)$  is majorized by a multiple of $x^{-\alpha}d\mu(x)$, also $x^{-1}d\mu(x)$ is determinate. By Theorem~\ref{thm:-1} $\mu$ must be the Friedrichs solution.
\end{proof}

Combining Theorem~\ref{thm:alphaS} with Theorem~\ref{thm:(1+x^2)^{-1/2}} we get:

\begin{cor}\label{thm:AB} Let $\varepsilon>0$ be arbitrary. There exists a normalized  indeterminate  Stieltjes moment sequence $s$ which is indet(S) such that $a(\mu_F)<\varepsilon$ while  $a(\mu_t)\ge 1/2$ for $F(s)<t\le \infty$.
\end{cor}

We shall now see that there exist Friedrichs solutions $\mu_F$ such that $a(\mu_F)$ is arbitrarily close to $1/2$.

\begin{prop}\label{thm:closehalf} For any $0<\delta<1/2$ there exists a Friedrichs solution
$\mu_F$ to an indet(S) Stieltjes moment problem  with $a(\mu_F)\in[1/2-\delta,1/2]$.
\end{prop}

\begin{proof} Assume for contradiction that the result does not hold. Then there exists
$0<\delta<1/2$ such that for all Friedrichs solutions $\mu$ to indet(S) Stieltjes moment problems we have $a(\mu)<1/2-\delta$. Starting from such a Friedrichs solution $\mu$ we shall construct another Friedrichs solution $\mu_1$  with $a(\mu_1)=\delta+a(\mu)$.

Since $a(\mu)<1/2-\delta<1-\delta$, we get by Proposition~\ref{thm:shift} that 
$\mu_1:=(1+x^2)^\delta d\mu(x)$ belongs to $\mathcal M_{i,0}^*$, has the same support as $\mu$  and satisfies $a(\mu_1)=\delta+a(\mu)<1/2$.
The last inequality shows that $\mu_1$ is a Friedrichs solutions and hence by our assumption
$a(\mu_1)=\delta+a(\mu)<1/2-\delta$, i.e., $a(\mu)<1/2-2\delta$. Repeating the argument leads to $a(\mu)<1/2-2^k\delta, k=1,2,\ldots$ which is not possible as $a(\mu)\ge 0$.
\end{proof}

\section{Examples of Friedrichs and Krein solutions}

\begin{ex}{\rm \textbf{Quartic rates.}
Transforming the moment problem for quartic birth and death rates in \cite{B:V} by the mapping $\tau_a(x)=ax$, where $a=(K_0/\pi)^4$, $K_0=\Gamma^2(1/4)/(4\sqrt{\pi})$,
we  get a normalized Friedrichs solution of the form
\begin{equation}\label{eq:Fr}
\mu_F=\frac{4\pi}{K_0^2}\sum_{k=0}^\infty \frac{(2k+1)\pi}{\sinh((2k+1)\pi)}\delta_{(2k+1)^4}.
\end{equation}
The corresponding Krein solution is
\begin{equation}\label{eq:Kr}
\mu_K=\frac{\pi}{K_0^2}\delta_0+\frac{4\pi}{K_0^2}\sum_{k=1}^\infty \frac{2k\pi}{\sinh(2k\pi)}\delta_{(2k)^4}.
\end{equation}
We know from Theorem~\ref{thm:(1+x^2)^{-1/2}} that $\mu_F\in B(\frac12), \mu_K\in A(\frac12)$.

Consider the measures
\begin{equation}\label{eq:c}
\mu_c=\sum_{k=0}^\infty \frac{(2k+1)\pi}{\sinh((2k+1)\pi)}(2k+1)^c\delta_{(2k+1)^4}.
\end{equation}

From Corollary~\ref{thm:Es}  we know that
$\mu_0, \mu_4$ are indeterminate and N-extremal, while $\mu_{-4}$ is determinate by Theorem~\ref{thm:-1}.

From this we clearly get that $\mu_c$ is indeterminate for $c\ge 0$ and  determinate for $c\le -4$. We are not able to decide for which value of $c\in (-4,0)$ $\mu_c$ changes from indeterminate to determinate.
}
\end{ex}

In the next examples we need the notation from the $q$-calculus of \cite{G:R}.
For a fixed number $0<q<1$ and $z\in\C$ we define
\begin{equation}\label{eq:qcalc}
(z;q)_0=1,\quad (z;q)_n=\prod_{k=0}^{n-1}(1-zq^k),\quad n=1, 2, \ldots, \infty.
\end{equation}
The Gaussian binomial coefficient is defined as
\begin{equation}\label{eq:Gauss}
\left[\begin{array}{ll} n\\k\end{array}\right]_q=\frac{(q;q)_n}{(q;q)_k(q;q)_{n-k}},\quad 0\le k\le n<\infty.
\end{equation}

\begin{ex}{\rm \textbf{Al-Salam--Carlitz moment problem.}
Assume that $0<q<1<a<1/q,$ which is the case of the Al-Salam--Carlitz moment problem in the indet(S) case, cf. \cite[Proposition 4.5.1]{B:V}.

The Friedrichs solution is given as
$$
\mu_F=(q/a;q)_\infty\sum_{k=0}^\infty \frac{a^{-k}q^{k^2}}{(q/a;q)_k(q;q)_k}\delta_{aq^{-k}-1},
$$
while the Krein solution is given as
$$ 
\mu_K=(aq;q)_\infty\sum_{k=0}^\infty \frac{a^{k}q^{k^2}}{(aq;q)_k(q;q)_k}\delta_{q^{-k}-1}.
$$
The Friedrichs parameter is given by
$$
F(s)=(q;q)_\infty\sum_{k=0}^\infty \frac{q^k}{(a-q^k)(q;q)_k},
$$
cf. \cite[formula (4.18)]{B:V}.
  We have $\mu_F\in B(\frac12)$ and $\mu_K\in A(\frac12)$.
}
\end{ex}

\begin{ex}{\rm \textbf{Stieltjes--Wigert moment problem.} We refer  to \cite{Chr} for this moment problem. For $0<q<1$ we have the density
$$
v_q(x)=\frac{q^{1/8}}{\sqrt{2\pi\log(1/q)}} x^{-1/2}\exp\left(-\frac{\log^2(x)}{2\log(1/q)}\right), x>0
$$
with moments
\begin{equation}\label{eq:SW}
s_n=\int_0^\infty x^n v_q(x)\,dx=q^{-n(n+1)/2}, n=0,1,\ldots,
\end{equation}
cf. \cite[formula (2.1)]{Chr}. 

To describe the Friedrichs and Krein solutions we need Ramanujan's function
\begin{equation}\label{eq:Ram}
\Phi(x)=\sum_{k=0}^\infty (-1)^k\frac{q^{k^2}}{(q;q)_k}x^k,
\end{equation}
The zeros of $\Phi$ form the sequence $\xi_1<\xi_2<\ldots$ described in Proposition~\ref{thm:revisit} as limits of the zeros of the orthonormal Stieltjes--Wigert polynomials
$$
p_n(x)=(-1)^n\sqrt{q^n/(q;q)_n}\sum_{k=0}^n \left[\begin{array}{ll} n\\k\end{array}\right]_q(-1)^k q^{k^2}x^k,
$$
cf. \cite[formula (2.4)]{Chr}. Detailed information about the zeros $(\xi_n)_{n\ge 1}$ can be found in \cite{I:Z}, where $\xi_n$ is denoted $i_n(q)$.

The Friedrichs parameter is given as $F(s)=1-(q;q)_\infty$, cf. \cite[formula (3.13)]{Chr}.
Three N-extremal solutions $\mu_t$ are known corresponding to the parameters $t=F(s)$ (the Friedrichs solution $\mu_F$), $t=1$ and $t=\infty$ (the Krein solution $\mu_K$)
concentrated in the sequences $(\xi_n)_{n\ge 1}$, $(q\xi_n)_{n\ge 1}$ and $0, (\xi_n/q)_{n\ge 1}$ respectively. 

It is worth recalling that the zeros $\xi_n$ are well separated in the sense that $\xi_{n+1}/\xi_n>q^{-2}$. 

We also recall that the mass at a  point $x_0\in\R$ of an N-extremal solution is given as $\rho(x_0)$, where
$$
1/\rho(x_0)=\sum_{n=0}^\infty p_n^2(x_0), 
$$
and therefore
$$
\mu_F=\sum_{k=1}^\infty \rho(\xi_k)\delta_{\xi_k},\quad \mu_1=\sum_{k=1}^\infty \rho(q\xi_k)\delta_{q\xi_k}
$$
and
$$
\mu_K=\rho(0)\delta_0+\sum_{k=1}^\infty \rho(\xi_k/q)\delta_{\xi_k/q}.
$$
We know that $\mu_F\in B(\frac12)$ and $\mu_1, \mu_K\in A(\frac12)$.

}
\end{ex}

\noindent
Christian Berg\\
Department of Mathematical Sciences, University of Copenhagen\\
Universitetsparken 5, DK-2100 Copenhagen, Denmark\\
e-mail: {\tt{berg@math.ku.dk}}

\vspace{0.4cm}
\noindent
Ryszard Szwarc\\
Institute of Mathematics, University of Wroc{\l}aw\\
pl.\ Grunwaldzki 2/4, 50-384 Wroc{\l}aw, Poland\\ 
e-mail: {\tt{szwarc2@gmail.com}}

\end{document}